\documentclass[a4paper,11pt,reqno]{amsart}
\usepackage{amsmath,amsthm,amssymb,mathabx}
\usepackage{mathtools}
\usepackage{bm}
\usepackage[usenames,dvipsnames]{xcolor}
\usepackage[shortlabels]{enumitem}
\allowdisplaybreaks 

\usepackage{fullpage}

\usepackage
[
  bookmarks = true, 
  pagebackref = true, 
  pdfauthor = {BBH}, 
  pdftitle = \text{K Biggs, J Brandes, K Hughes - Reinforcing a philosophy}, 
  colorlinks = true, 
  linkcolor = blue, 
  citecolor = blue]
{hyperref}

\newcommand{\C}{\mathbb{C} }
\newcommand{\R}{\mathbb{R} }
\newcommand{\Z}{\mathbb{Z} }
\newcommand{\T}{\mathbb{T} }
\newcommand{\Q}{\mathbb{Q} }
\newcommand{\N}{\mathbb{N} }

\newcommand{\F}{\mathbb{F}}

\newcommand{\lf}{K} 
\newcommand{\roi}{\mathfrak{o}} 
\newcommand{\uniformizer}{\pi} 
 
\newcommand{\rroi}{\mathbb{Z}_p} 

\newcommand{\partition}{\mathcal{P}} 
\newcommand{\scales}{\mathcal{R}} 

\newcommand{\oddprime}{\mathfrak{p}}
\newcommand{\mydegree}{k}

\newcommand{\eof}[1]{e({#1})}

\newcommand{\vof}[1]{{\boldsymbol #1}}
\newcommand{\trace}[2]{{\Tr_{#1}(#2)}}

\def\dd{{\,{\rm d}}}
\def\th{{{\rm th}}}

\DeclareMathOperator{\dist}{dist}
\DeclareMathOperator{\Tr}{Tr}
\DeclareMathOperator{\characteristic}{char}
\DeclareMathOperator{\Arg}{Arg}
\DeclareMathOperator{\rp}{Re}
\DeclareMathOperator{\ip}{Im}

\newtheorem{theorem}{Theorem}[section]
\newtheorem{lemma}[theorem]{Lemma}
\newtheorem{proposition}[theorem]{Proposition}
\newtheorem{corollary}[theorem]{Corollary}

\newtheorem*{Rolle}{Rolle's Theorem}
\newtheorem*{Lagrange}{Lagrange Interpolation}

\theoremstyle{definition}

\theoremstyle{remark}
\newtheorem{remark}[theorem]{Remark}

\numberwithin{equation}{section}

\newcommand{\mmod}[1]{\,\,(\text{mod}\,\,#1)}

\usepackage[foot]{amsaddr}

\title{\vspace*{-1.5cm}Reinforcing a Philosophy:\\ A counting approach to square functions over local fields}
\author{\vspace{-2mm}Kirsti D. Biggs$^{\ast}$, Julia Brandes and Kevin Hughes}

\address{KB: Department of Mathematics, KTH Royal Institute of Technology, 10044 Stockholm, Sweden.}
\email{kirstib@kth.se}

\address{JB: Department of Mathematical Sciences, University of Gothenburg and Chalmers University of Technology, 41296 Gothenburg, Sweden.}
\email{brjulia@chalmers.se}

\email{khughes.math@gmail.com} 

\address{$^{\ast}$Corresponding Author}

\dedicatory{This work is dedicated to the memory of Prof. Lynne H. Walling, to whom each of the authors is deeply indebted. Her strong and supportive personality will be sorely missed.}

\begin{document}

\begin{abstract}
In this paper, we study square functions for extension operators over finite-type, planar curves endowed with the Euclidean arclength measure. 
We prove new results for curves of the form $(T,\phi(T))$ where $\phi(T)$ is a polynomial of degree at least 2. This includes new estimates for such curves given by monomials $\phi(T) = T^k$ for $k \geq 3$ which are uniform over all local fields whose characteristic is coprime to \(k\). 
Key to our approach is a systematic analysis of the second order differencing polynomial and its geometry in local fields. 
\end{abstract}

\maketitle

% % % 
\section{Introduction}
% % % 

% % 
\subsection{Background}
% % 

In the intersection of arithmetic and harmonic analysis there is a hierarchy of operators. In increasing strength, these begin with mean value estimates from analytic number theory, which are implied by discrete restriction estimates, which are in turn implied by decoupling estimates, and finally with square functions at the top of the hierarchy whose bounds imply those below. Though all of these phenomena differ, a natural question is to examine how close they are to each other. To close the loop, it suffices to ask: \emph{In what way and under what conditions can methods for counting solutions to Diophantine equations be adapted to deduce square function estimates?}  

This hierarchy has been examined by many researchers from both analytic and number theoretic vantage points. From the number theory viewpoint, Wooley's development of his efficient congruencing method in \cite{Wooley:cubic, Wooley:discrete, Wooley:NEC} introduced the philosophy that arithmetic methods could be adapted to climb this hierarchy. 
This was in contrast to the work \cite{BDG} on decoupling for the moment curve, which proved the Vinogradov Mean Value Conjectures through proving the associated sharp decoupling inequality by purely analytic methods, thus reinforcing the analytic viewpoint of directly proving sharp decoupling estimates from which the desired number theory estimates swiftly follow.

The work at hand was motivated by \cite{GGPRY} which established square function estimates for non-degenerate curves in $\R^n$ by counting near-solutions to the associated system of underlying equations. Their main point was that the diagonal behavior underlying the Vinogradov system of equations can be strengthened to count near-solutions, and they used this behavior to upgrade the classical estimate for the number of solutions to a square function estimate. 
This gave an elegant solution to the question in the first paragraph for non-degenerate, real curves. 

Against this backdrop, the purpose of the work here is to reinforce and expand this number theoretic philosophy by proving square function estimates in two new settings. The first setting is for degenerate curves and the second is over more general local fields. 
By providing an argument that is agnostic with respect to the ground field, we illustrate the universality of a counting approach.

% % 
\subsection{An indication of our results}
% % 

Due to our setting over general local fields, stating our results in their full generality is somewhat notation-heavy, so we defer this to Section \ref{section:setup}. Presently, we will describe a version of our result over $\R$ in order to give the reader a rough idea of our main theorem. 

Let $\phi$ denote a continuous function over $\R$, and consider the planar curve $\gamma(T) := (T,\phi(T))$ in $\R^2$. When $I \subset \R$ is measurable, we define the \emph{extension operator on $\gamma$ over $I$} via its action on a continuous and compactly supported function $f$ by putting   
\[
E_{I}f(\vof{x}) 
:= 
\int_{I} f(\xi) \, e^{-2\pi i \langle \gamma(\xi),\vof{x} \rangle} \dd{\xi} 
\quad \text{for points} \quad \vof{x} \in \R^2.
\]
Next, let $\partition_{R^{-1}}([-1/2,1/2])$ be a partition of the interval $[-1/2,1/2]$ into adjacent subintervals of length $R^{-1}$ where $R\in\N$. We then define the \emph{square function at scale} $R \in \N$ as 
\[
S_{\partition_{R^{-1}}([-1/2,1/2])} f(\vof{x}) 
:= 
\big( \sum_{J \in \partition_{R^{-1}}([-1/2,1/2])} |E_{J}f(\vof{x})|^2 \big)^{1/2}
.\]

In this notation, our main result over $\R$ is as follows.  
\begin{theorem}\label{theorem:real:2D}
Let $k \geq 2$. Suppose that $\gamma(T) := (T,\phi(T))$ with $\phi(T) \in C^{\mydegree}([-1/2,1/2])$ is such that $|\partial^{\mydegree}{\phi}(t)| \neq 0$ for all $t$ in $[-1,1]$. 
There exists a positive constant $C_{\phi}$, depending only on $\phi$, such that for all $R \in \N$ and all balls $B$ of diameter at least $C_{\phi} R^\mydegree$ in $\R^2$, we have the inequality 
\begin{equation}\label{inequality:real:2D}
\| E_{[-1/2,1/2]} f \|_{L^4(B)} 
\leq 5\sqrt[4]{\mydegree} \; 
\| S_{\partition_{R^{-1}}([-1/2,1/2])} f \|_{L^4(B)} 
.\end{equation}
\end{theorem} 

The case $\mydegree=2$ of Theorem~\ref{theorem:real:2D} corresponding to non-degenerate planar curves is attributed to Fefferman \cite{Fefferman} by \cite{GGPRY}. For the \(k=2\) case, the proofs in \cite{Fefferman, GGPRY} use very different methods from each other and the one employed in our paper, but the heart of the matter in all proofs is the underlying diagonal behavior of the system.\footnote{
We note that the square functions in \cite{Fefferman} are those related to Bochner--Riesz operators which differ from the square functions in \cite{GGPRY}. The two types of square functions are closely related, and experts can readily translate the proof in \cite{Fefferman} for Bochner--Riesz operators to one for the square functions herein when \(k=2\).
} 
Theorem~\ref{theorem:real:2D} extends the two-dimensional results of \cite{GGPRY} to sufficiently smooth, not completely flat, real planar curves with vanishing curvature. A good example here is the curve $(X,X^3)$ whose curvature vanishes at the origin. 

While some version of Theorem~\ref{theorem:real:2D} may be known (e.g., for square functions associated to Bochner--Riesz means along the moment curve in three or more dimensions as in \cite{P1, P2}, mollifications by arclength measure, or containing an \(\epsilon\)-loss), we have not found this result in the literature. We also emphasize that our results do not harbor an \(\epsilon\)-loss\footnote{An \(\epsilon\)-loss for \eqref{inequality:real:2D} would be an additional factor of \(C_\epsilon R^\epsilon\) where \(C_\epsilon\) depends on \(\epsilon>0\).}, which differs crucially from related endpoint restriction results for planar curves. Furthermore, we extend Theorem~\ref{theorem:real:2D} in the following ways, each of which is new:
\begin{itemize}
\item for polynomial curves, Theorem~\ref{theorem:main:2D} obtains a version uniform over generic local fields;
\item when \(\mydegree=2\), Theorem~\ref{theorem:convex} weakens the \(C^2[-1/2,1/2]\) assumption to strictly convex or strictly concave;
\item Theorem~\ref{theorem:complex} obtains a version of our result for complex analytic curves. 
\end{itemize}
We defer the statements of each theorem to subsequent sections. 

Of independent interest is our systematic analysis of the second order differencing polynomial (defined in \eqref{def:SOD}) which lies at the heart of our approach. This removes any need for Kakeya-type information. We have organized our paper to make transparent its universal role across local fields and to illuminate its relevant properties in each setting. 
We pay attention to the second order differencing polynomial for non-degenerate planar curves and examine its geometry in local fields. Interestingly, there is a stark difference in its geometry between Archimedean and non-Archimedean fields. This appears to reflect the different properties of derivatives in these types of local fields.

% % 
\subsection{Overview of our approach: The Second Order Differencing Method}
% % 

Initially, our goal was to extend \cite{GGPRY} to arbitrary local fields. The critical ingredient in their approach was their establishment of `diagonal behavior' in near-solutions to the underlying system of equations. Unfortunately, one cannot adapt the techniques of \cite{GGPRY} to non-Archimedean local fields because of their use of the Mean Value Theorem in the proof of their Lemma~4.3, which feeds into the diagonal behavior fundamental to their approach. The Mean Value Theorem is an inherently real phenomenon that fails in non-Archimedean local fields. Consequently, a different approach must be found. 

In order to recover, one falls back to considering the prototypical non-degenerate (meaning without vanishing torsion) real curve which is the moment curve. Fortunately, the moment curve is well-defined over generic fields, and in this case Proposition~4.1 of \cite{GGPRY} follows from classical properties of Vandermonde matrices. The analogue of Proposition~4.2 of \cite{GGPRY}, however, remains mysterious. Diagonal behavior for the moment curve is well-known in the circle method, which gives some hope of proving the analogous diagonal behavior encoded by Proposition~1.3 of \cite{GGPRY}. 
In practice, for the moment curve over generic local fields, one can circumvent Proposition~4.2 of \cite{GGPRY} and simply prove the required diagonal behavior using the Newton--Girard equations, after which boundedness of the square function analogous to \eqref{inequality:real:2D} follows. 

For more on the approach alluded to above for the moment curve over generic local fields, see \cite{H}; specifically, see Theorem~1 therein for a precise bound and Proposition~6 for the associated diagonal behavior. 
A better upper bound for the analogue of Theorem~\ref{theorem:main:2D} for the moment curve in generic local fields was proved in \cite{WH}. 
The upper bound of \cite{WH} turns out to be sharp as proved in \cite{H}; see Theorem~2 in the latter reference. Both papers crucially rely on diagonal behavior for the moment curve. 
We also mention that only the moment curve is treated in \cite{WH}. Over non-Archimedean local fields, it appears to us that the notion of torsion is rather subtle because of its reliance on derivatives. In turn, this makes it difficult to define non-degenerate curves over generic local fields and establish a general result for them such as \cite{GGPRY} furnished for real curves. This is explored in Section~\ref{section:Rolle}.

With the moment curve in hand, we turned our attention to degenerate (meaning with vanishing torsion), finite-type (meaning some derivative is non-vanishing) curves. For finite-type curves, prior experience with the curve \((X,X^3)\) revealed another obstacle: \emph{there are off-diagonal solutions to the underlying systems of equations}. Similarly, there are off-diagonal near-solutions to the underlying systems of approximate equations. Our first task was to understand the structure of off-diagonal, near-solutions in each local field. In turn, this led us to Wooley's Second Order Differencing Method, which we review in a moment, and to the mantra of `special subvarieties' which captures the structure we studied. This method is the main technique in our paper and its use appears to be novel in these problems. 

The second obstacle was that the approach of \cite{GGPRY} was only capable of handling diagonal behavior; see the paragraph following (1.11) on page 7080 of \cite{GGPRY} where they emphasize this point. To be precise, their reduction to near-solutions of systems of equations relied on having essentially only diagonal behavior. After understanding the structure of off-diagonal near-solutions, the reduction from square function estimates to special subvarieties required some finesse utilizing the symmetry of the underlying systems of equations.

Our main observation is the fact that special subvarieties play a role for square-function estimates on curves of finite-type. From the point of view of Diophantine equations, this indicates that rather than measuring `strong diagonal behavior', square function estimates should be viewed as a sign that the number of rational points on the underlying variety accumulates on a finite number of special subvarieties of the appropriate dimension. Here, the set of diagonal solutions is the archetypical example of such a subvariety, but there may be more. The easiest example of this is the curve $\gamma(T)=(T,T^3)$, where the square function estimate is connected to the problem of counting solutions to the pair of equations  
\begin{align}
\begin{split}\label{eq:DS:KdV}
x_1^3 + x_2^3 &= y_1^3+y_2^3, 
\\ x_1+x_2 &= y_1+y_2.
\end{split}
\end{align}
It is not hard to see that beyond the `strongly diagonal' solutions $\{x_1, x_2\}=\{y_1,y_2\}$ there are also solutions of the shape $x_1=-x_2$, $y_1=-y_2$.

Our main tool for capturing such special subvarieties is Wooley's \emph{Second Order Differencing Method}; we refer to this as `Wooley's S.O.D.ing method' for short. Suppose that $K$ is a field. For any polynomial $\phi \in \lf[X]$ we define the associated second order differencing polynomial $\psi_{\phi}$  via the relation 
\begin{equation}\label{def:SOD}
\phi(X+Y-Z)-\phi(X)-\phi(Y)+\phi(Z) = (X-Z)(Y-Z)\psi_{\phi}(X,Y,Z) 
.\end{equation}
In this factorization, the factors $X-Z$ and $Y-Z$ correspond to the two planes $x_1=y_1$ and $x_2=y_1$, respectively, which give the strong diagonal solutions. Meanwhile, off-diagonal solutions correspond to zeros of the function $\psi_{\phi}$.
We will typically refer to the trivariate polynomial $\psi_{\phi} \in K[X,Y,Z]$ as the `S.O.D.ing polynomial' associated with $\phi$. 
When the context is clear, we will drop the subscript \(\phi\) and  write \(\psi\) in place of \(\psi_{\phi}\). 

The existence of an identity of the shape given in \eqref{def:SOD} is classical when $\phi(X)=X^3$, in which case $\psi_\phi(X,Y,Z)=3(X+Y)$, and it has left its mark in arguments in analytic number theory at least since the work of Hua \cite{Hua1965} on Waring's problem\footnote{See Chapter~5; specifically, p.40 for the definition of \(S_k\) and Theorem~8, followed by Lemma~5.2 on p.43 which establishes a sixth moment estimate for \(S_3\), and the bottom of p.47 which establishes the desired tenth moment estimate.} and in analysis since at least the work of Bourgain \cite{Bourgain:KdV} on the periodic KdV equation\footnote{See eq.(8.15) and how it is immediately used in Lemma~8.16. A similar technique is used in eqs.(8.37)--(8.44) on pp.227--228.}. 
Beyond the case of $X^3$, S.O.D.ing polynomials and their analysis have been used to great effect in the work of Wooley: see in particular pages~158--159 in \cite{Wooley:simultaneous} as well as subsequent works \cite{Wooley:symmetric,VW:nonary,PW}. 
Recently, the method was introduced and exploited in harmonic analysis in \cite{HW} to prove an \(\ell^2(\Z) \to L^{10}(\T^2)\) discrete adjoint restriction estimate for the curve \((X,X^3)\). In this example, the corresponding putative decoupling inequality in the hierarchy is known to fail, and this result shows that the hierarchy discussed above is not always applicable. 

Wooley's S.O.D.ing method was developed in \cite{DHV} to prove new sharp discrete \(\ell^p\)-improving estimates for curves of the form considered in this paper. 
This paper further develops Wooley's S.O.D.ing method by analyzing approximate versions of the system of equations in \eqref{eq:DS:KdV} and its generalizations. This is the key ingredient in all of our proofs. 
We complete our study by exploring the limitations of this method and contrasting its implementation over various local fields.

% % 
\subsection{Outline of the paper}
% % 

In the propreantepenultimate section (Section~\ref{section:setup}) we set several notations and conventions used for local fields. 
In the preantepenultimate section (Section~\ref{section:main_prop}), we introduce the main proposition used in proving our results; this proposition relates approximate versions of \eqref{eq:DS:KdV} to special subvarieties associated to the curve. 
In the antepenultimate section (Section~\ref{section:CauchySchwarz}), we use our main proposition to prove our main theorem, Theorem~\ref{theorem:main:2D}. 
In the penultimate section (Section~\ref{section:real}), we prove Theorem~\ref{theorem:real:2D}. 
In the ultimate section (Section~\ref{section:Rolle}), we discuss the $p$-adic geometry of second order differencing polynomials.

% % 
\subsection*{Acknowledgements}
% % 

The third author thanks Trevor Wooley for introducing him to the method and for providing valuable references. 
The third author also thanks Jim Wright for inspiring his interest in $p$-adic harmonic analysis during 2013 and 2014 and sharing his contributions to the area.

% % 
\subsection*{Declarations}
% % 

\emph{Funding}. 
During the production of this manuscript, the first author was supported by postdoctoral grant no.~2018.0362 of the Knut and Alice Wallenberg Foundation, in the name of the second author. In addition, the second author held a Starting Grant of the Swedish Science Foundation (Vetenskapsr{\aa}det) under grant agreement no.~2017-05110, and in the later stages a Project Grant (no.~2022-03717) by the same funder. We are grateful to both bodies for allowing the third author to visit Gothenburg in February 2020 and in November 2021, as well as to the department for mathematical sciences at Chalmers University of Technology and the University of Gothenburg for its hospitality. 
The third author was supported by the Additional Funding Programme for Mathematical Sciences, delivered by EPSRC (EP/V521917/1) and the Heilbronn Institute for Mathematical Research.
Part of the work was undertaken while the authors were participating in the Harmonic Analysis and Analytic Number Theory Trimester Program at the Hausdorff Institute of Mathematics in the summer of 2021, whose hospitality is also gratefully acknowledged.

\emph{Competing interests}. 
The authors have no relevant financial or non-financial interests to disclose.

% % % 
\section{Set-up for local fields}\label{section:setup}
% % % 

In this paper, we let $\lf$ denote a (one-dimensional) local field: in other words, a locally compact topological field which is complete with respect to a non-discrete topology. 
We recall only what we need for this work and broadly refer the interested reader to Chapters I, II and XV of \cite{CasselsFroehlich} for more information about local fields and proofs of various facts that we recall. 

The topology of $\lf$ is induced by the metric associated to an absolute value $|\cdot|_\lf : \lf \to \R_{\geq0}$. When $\lf$ is fixed we will suppress the dependence on $K$ in our notation for its absolute value. We extend the metric to $\lf^2$ by defining $|(x_1,x_2)|_{\lf} = \max\{|x_1|,|x_2|\}$ for $(x_1,x_2) \in \lf^2$. So, geometrically all of our `balls' are `squares', and we use these terms interchangeably. 
Similarly, we refer to the diameter of a ball or square $B$ in $\lf^2$ as $\sup_{x,y \in B} |x-y|$. This notion of diameter coincides with `side-length' when the local field is Archimedean and `radius' when the local field is non-Archimedean. 

Local fields come in two flavors: each one is either a completion of a finite extension of $\Q$ with respect to a non-trivial metric given by a absolute value, or it can be realized as the set of formal Laurent series over a finite field. 
Fields of the former flavor come in two types. The first type are typically referred to as \emph{Archimedean} and are isomorphic to a finite field extension of $\R$. The second type are called \emph{non-Archimedean} and correspond to finite field extensions of $\Q_p$ for a rational prime $p$. Note that the only possible Archimedean local fields are $\R$ and $\C$, whereas non-Archimedean fields can have arbitrary degree. 

Each field of formal Laurent series over a finite field has infinitely many non-Archimedean valuations for which it is complete with respect to the induced metric. In other words, every field of formal Laurent series over a finite field is non-Archimedean. When discussing such a field of formal Laurent series, we implicitly choose one of its non-Archimedean valuations. Our results are uniform over any choice of valuation.

For a field extension $L/K$, we write $[L:K]$ for its degree. For each field $\lf$, we will work with the ring of univariate polynomials $\lf[X]$ and the ring of univariate power series $\lf[[X]]$. The latter is formally defined as the set of all infinite power series of the form $a_0+a_1X+a_2X^2+\dots$ such that $a_0,a_1,a_2,\dots \in \lf$. Multiplication and addition of power series follows the rules of multiplication and addition of polynomials. 

Over each field, we have the differentiation operator $\partial: K[[X]] \to K[[X]]$ defined formally by its action on a power series as 
\[
\partial: \sum_{i=0}^{\infty} a_i X^i \mapsto \sum_{i=0}^{\infty} (i+1) a_{i+1} X^i.
\]
In this work, we will be mostly concerned with the case in which $\phi$ is a polynomial over $\lf[X]$, although occasionally we also consider more general functions $\phi$ over $\R$, in which case we have an associated Taylor series in $\R[[X]]$. For a formal power series $\phi \in \lf[[X]]$ we also write $\partial \phi = \phi'$, $\partial \partial \phi = \phi''$ and so on, as is standard practice over $\R$. In a few instances we will need to be precise with respect to which  variable we are differentiating, this will be indicated by $\partial_{\#}$, where $\#$ denotes the variable or coordinate with respect to which we differentiate. For instance, for a trivariate function $\psi(X,Y,Z)$, we write $\partial_Y \psi(X,Y,Z)$ for the partial derivative in the $Y$-variable.

% % 
\subsection{Non-Archimedean local fields}
% % 

For a non-Archimedean local field $\lf$, let $\roi := \{ x \in \lf : |x|_{\lf} \leq 1 \}$ denote its ring of integers and $\oddprime := \{ x \in \lf : |x|_{\lf} < 1 \}$ denote the unique maximal ideal of $\roi$. The residue field of $\lf$ is $\roi/\oddprime$, which is a finite field. Let $p = \characteristic (\roi/\oddprime)$. 
The image of $|\cdot|_{\lf}$ is cyclic and $\lf$ comes with a uniformizing element, say $\uniformizer \in \lf$, generating this group. 
We assume that the metric on the non-Archimedean local field $\lf$ is normalized such that $|\uniformizer|=p^{-1}$. 
For a scale $R$ in $\scales(\lf) := \{ p^s : s \in \Z_{\geq0} \}$, let $\partition_{R^{-1}}(\lf)$ denote a partition of $\roi$ into closed (and hence also open) balls of diameter $R^{-1}$. Observe that $\#\partition_{\delta}(\lf)$ is finite for each $\delta>0$. 

We define the additive character on a non-Archimedean local field $K$ as follows. Consider first the case when $K=\Q_p$ for some rational prime $p$. Any $x \in \Q_p$ can be written in the form $x=p^v a/b$ for some units $a,b \in \Z_p \setminus p \Z_p$ and an integer $v$. In this notation, we define $\eof{x}:=e^{-2\pi ip^v a}$, so that in particular we have $\eof{x}=1$ for all $x \in \Z_p$. For a general non-Archimedean field $K$ arising as a finite extension of $\Q_p$, we define the character $\eof{\cdot}$ on $\lf$ by $\eof{x}:=\eof{\trace{\lf/\Q_p}{x}}$, where the latter character is defined on $\Q_p$. 
The Haar measure $\dd{\xi}$ on $\lf$ is normalized so that the measure of $\roi$ is 1. 

Consider next non-Archimedean fields $K=\mathbb F_q((T))$ which arise as the function field over a finite field $\mathbb{F}_q$ of characteristic $p$, where $p$, as usual, is a finite rational prime. 
The character on $\lf$ is defined for a Laurent series 
\[
\alpha=\alpha(T):=\sum_{n=N}^\infty a_n T^n
\]
by setting $\eof{\alpha} := e_q(a_{-1})$, where $e_q$ is a fixed additive character on $\mathbb{F}_q$. 
As before, the Haar measure $\dd{\xi}$ on $\lf$ is normalized so that the measure of $\roi$ is 1.

% % 
\subsection{Archimedean local fields}
% % 

For an Archimedean field $\lf$, either $\lf$ is $\R$ or $\C$ in which case we have the absolute values $|\cdot|_\R$ on $\R$ with its standard Archimedean metric and $|z|_\C := \max\{|x|_\R,|y|_\R\}$ for $z=x+iy \in \C$. For either of these fields, we proceed by defining $\roi := \{ x \in K : |x|_{K} \leq 1/2 \}$. In these cases, $\roi$ is not the ring of integers, but instead the closure of a fundamental domain for the quotient of $\lf$ by its ring of integers (which are the usual integers in $\R$ and the Gaussian integers $\Z[i]$ in $\C$). In contrast, \cite{GGPRY} chose the interval $[0,1]$ in $\R$. Our choice here is motivated by our desire to wholly include points with vanishing torsion at the origin, but the two choices are equivalent by translation. 

We set $\scales(\R)=\scales(\C)=\mathbb N$.  For $R \in \scales(\R)$, define the partitions 
\[
\partition_{R^{-1}}(\R) 
:= \bigcup_{j=0}^{R-1}\left[-\frac12+\frac{j}{R},-\frac12+\frac{j+1}{R}\right) 
.\]
Similarly, for $R \in \scales(\C)$ we define  
\begin{align*}
\partition_{R^{-1}}(\C) 
&:= \bigcup_{j,k=0}^{R-1} \left\{\left[-\frac12 +\frac{j}{R},-\frac12+\frac{j+1}{R}\right) + i\left[-\frac12+\frac{k}{R},-\frac12+\frac{k+1}{R}\right)  \right\} 
.\end{align*}
We stress that while our partitions are defined over half-open, half-closed sets, our finite type conditions are assumed over compact sets $\roi$. The edges not considered in our partition form a zero measure set in the integrals and are unimportant in their estimation. 

On $\R$ the additive character is given by $\eof{t} := e^{-2\pi it}$. On $\C$, we extend this definition as above by defining $\eof{z} := \eof{\trace{\C/\R}{z}/2} = \eof{ \rp z}$. In either case, we write $\dd{x}$ for the ordinary Lebesgue measure.

% % 
\subsection{Smooth localization}
% % 

For any function $f \in L^1(K)$ where $\lf$ is a local field, its Fourier transform is defined as 
\[ 
\widehat{f}(\xi)
=
\int_{K}f(x)\eof{x\xi}\dd{x} 
.\] 
Suppose that $w:\R \to \R$ is a Schwartz function which is non-negative, at least 1 on the interval $[-1,1]$, and for which $\widehat{w}$ is supported on $[-1,1]$. Such a function is not difficult to construct; see for instance the discussion following (6.1) in \cite{GGPRY}. This function can be extended to $\C$ by putting $w(x+iy)=w(x)w(y)$. Meanwhile, on a non-Archimedean field $\lf$, we take $w:\lf \to \lf$ to be the indicator function on $\roi$. In each case, we define the two-dimensional Schwartz function $W: K^2 \to \R_{\geq0}$ as $W(x_1,x_2) := w(x_1)w(x_2)$. Thus, $W(x_1,x_2) \geq 1$ on the square $|(x_1,x_2)|\leq 1$, and its Fourier transform $\widehat{W}$ is supported on the square $|(x_1,x_2)|\leq 1$. In particular, when $K$ is a non-Archimedean field, both $W$ and its Fourier inverse $\widehat{W}$ are simply the indicator function on $\roi^2$. 

In each local field we define, for a ball $B$ in $\lf^2$ of diameter $R>0$ and center $\vof{c}$, the function 
\[ W_B(\vof{x}) := W\big(\frac{\vof{x}-\vof{c}}{R}\big) \]
for all $\vof{x} \in \lf^2$. This localizes $W$ to the square $B$ so that its Fourier transform is supported in a ball of diameter $R^{-1}$ about the origin. 
If $B$ is a square of diameter $R>0$ with center the origin in $\lf^2$, then 
\begin{equation}\label{uncertainty}
\int_{\lf^2} \eof{\vof{y}\cdot\vof{x}} W_{B}(\vof{x}) \dd{\vof{x}} 
= 
R^{2[\lf:F]} \widehat{W}(R\vof{y})
.\end{equation}
\noindent 
The power of $2[\lf:F]$ stems from the fact that we are considering a 2-dimensional integral over the field $K$, which in turn can be viewed as a vector space of dimension $[\lf:F]$ over its ground field $F$, where $F=\R, \Q_p$ or $\F_p((T))$ depending on whether $K$ is an Archimedean, a $\oddprime$-adic, or a function field. 

\begin{remark}
Over $\R$ and $\C$, our results in Section~\ref{section:real} and Theorem~\ref{theorem:real:2D} are stated for $L^p(B)$-norms where $B$ is a ball in $\R^2$. There are standard techniques to deduce square function inequalities for the $L^p(B)$-norm from those for the $L^p(W_B)$. Since these techniques are standard, we prove only the latter. The reader may consult the discussion surrounding (6.1) of \cite{GGPRY} for more information on such techniques. 
\end{remark}

% % 
\subsection{Statement of the main theorem}\label{section:maintheorem}
% % 

Let $\gamma(T) := (T,\phi(T))$ denote a curve, where $\phi$ is a continuous function on $\lf$. For each measurable set $I$ in $\lf$, define the \emph{extension operator on $\gamma$ over $I$} via its action on a continuous, compactly supported function $f$ by putting
\[
E_{I}f(\vof{x}) 
:= 
\int_{I} f(\xi) \eof{\gamma(\xi) \cdot \vof{x}} \dd{\xi}.
\] 
For a measurable set $U$ in $\roi$, we define the \emph{square function over $U$ at scale $R \geq 1$} as 
\[
S_{\partition_{R^{-1}}(U)} f(\vof{x}) 
:= 
\big( \sum_{J \in \partition_{R^{-1}}(U)} |E_{J}f(\vof{x})|^2 \big)^{1/2}
.\]

Our main theorem is the following. 
\begin{theorem}\label{theorem:main:2D}
Let $\lf$ be a local field. 
Suppose that $\gamma(T)$ is a curve of the form $(T,\phi(T))$ where $\phi(T) \in \lf[T]$ has degree $k \geq 2$ 
 not divisible by the characteristic of $\lf$. 
There exists a positive constant $C_{\phi}$ such that for all $R \in \scales(\lf)$ and all balls $B$ of diameter at least $C_{\phi} R^k$ in $\lf^2$, we have the inequality  
\begin{equation}\label{inequality:main:2D}
\| E_{\roi} f \|_{L^4(W_B)} 
\leq 5\sqrt[4]{\deg{\phi}} \; 
\| S_{\partition_{R^{-1}}(\roi)} f \|_{L^4(W_B)} 
.\end{equation}  
\end{theorem}

An instructive example for this theorem is provided by the class of monomial curves $(X,X^k)$ for $k \geq 2$. In every local field, these have a flat point at the origin when $k \geq 3$, and thus the argument of \cite{GGPRY} fails in these instances. However, the origin is not infinitely flat, and this is fundamental to obtain our theorem. 
For monomial curves, and more generally any curve such that $\phi \in \Z[T]$, a pleasant feature of Theorem~\ref{theorem:main:2D} is that the bounds therein are uniform over any embedding of the curve in $\R$, $\C$ or a non-Archimedean field $\lf$. 

\begin{remark}
The condition that the degree is not divisible by the characteristic of the field is necessary. To see this, consider for instance the curve $(X,X^3)$ over a function field of characteristic 3, where we find that $(X+Y-Z)^3-(X^3+Y^3-Z^3)=0$, so that the associated S.O.D.ing polynomial $\psi(X,Y,Z)$ vanishes identically. Similar degenerate cases arise for each finite characteristic from the linearity of the Frobenius map.
\end{remark}

\begin{remark}
A careful inspection of the proof reveals that the factor $5$ can be replaced by $\sqrt 5$ when $K=\R$, and by $1$ when $K$ is non-Archimedean.
\end{remark}

\begin{remark}
The proof of the theorem generalizes to handle curves $\gamma(T) = (T,\phi_2(T),\dots,\phi_r(T))$ on $L^4(B)$ where $B$ is a ball of diameter at least $R^k$ in $\lf$, provided that the set $\{1, T, \phi_i(T)\}$ is linearly independent for some $i$; here $k$ is the degree of $\phi_i$. The reason for this is that while these curves may lack torsion, the curves have torsion when projected onto the two-dimensional subspace spanned by the first and $i^{\textrm{th}}$ coordinates. This can be seen as a weakening of the traditional notion of torsion, but it turns out to be sufficient for $L^4$ norms. Whilst this is a well-known phenomenon in number theory, we are unaware of the exploration of this type of behavior in harmonic analysis. We expect that a corresponding condition, where torsion is required only for a lower-dimensional subspace of an appropriate dimension, will be sufficient for higher $L^p$-norms also. A similar phenomenon arose in recent work by the second and third authors \cite{BH}. 
\end{remark}

\begin{remark}
When the local field is $\C$ interpreted as $\R^2$, our estimates take an interesting shape because we may interpret the complex curves $(X,\phi(X))$ as surfaces in $\R^4$. For instance, 
the complex curves associated with $\phi(z)=z^2$ and $\phi(z)=z^3$ give rise to square function estimates for the real surfaces $(X, Y, X^2-Y^2,2XY)$ and $(X, Y, X^3-3XY^2, 3X^2Y-Y^3)$, respectively.
\end{remark}

% % % 
\section{The main proposition}\label{section:main_prop}
% % % 

Fix the local field $\lf$ and the curve $\gamma(T) = (T,\phi(T))$ where $\phi \in \roi[T]$. Let $k \in \N$ be the degree of $\phi$. Our general strategy is to follow \cite{GGPRY}. 
However, their proof does not apply to curves with vanishing torsion. In particular, the analogue of Proposition~1.3 in \cite{GGPRY} fails for curves with vanishing torsion. Fortunately, `the paradigm of special subvarieties' furnishes the following surrogate. 
\begin{proposition}\label{prop:KdV}
Let $\lf$ be a local field. 
Suppose that $\phi(T) \in \lf[T]$ has degree $k \geq 2$, and that $R \in \scales(\lf)$. 
Fix $I_1,I_1'$ in $\partition_{R^{-1}}$ such that $\dist(I_1,I_1') \geq R^{-1}$. 
There exists a positive constant $C_{\phi}$ such that there are at most $625({\deg{\phi}}-1)$ pairs $\{I_2,I_2'\}$ of intervals in $\partition_{R^{-1}}$ with the property that for all points $t_1 \in I_1$, $t_2 \in I_2$, $t_1' \in I_1'$ and $t_2' \in I_2'$, we have 
\begin{equation}\label{close_points:KdV}
|\gamma(t_1)+\gamma(t_2)-\gamma(t_1')-\gamma(t_2')| 
\leq C_{\phi} R^{-k}
.\end{equation}
\end{proposition}

\begin{proof}

Suppose that \eqref{close_points:KdV} is satisfied with some constant $C_{\phi}$ which we will determine later. We then have the linear inequality $|t_1+t_2-t_1'-t_2'| \leq C_{\phi} R^{-k}$. Put $\delta=t_1+t_2-t_1'-t_2'$, so that $|\delta| \leq C_{\phi} R^{-k}$. 
We use this to substitute the variable $t_2'$ within the higher degree equation to find 
\[
|\phi(t_1)+\phi(t_2)-\phi(t_1')-\phi(t_1+t_2-t_1'-\delta)| 
\leq C_{\phi} R^{-k}
.\]
By the Binomial Theorem, we have 
\begin{equation}\label{approx:Taylor}
\phi(t_1+t_2-t_1'-\delta) 
= 
\phi(t_1+t_2-t_1') + \sum_{j=1}^k\frac{ (-\delta)^j}{j!} \phi^{(j)}(t_1+t_2-t_1')  
,\end{equation}
so that 
\[
|\phi(t_1+t_2-t_1'-\delta) -\phi(t_1+t_2-t_1')| \lesssim_{\lf,\phi} |\delta|. 
\]
Inserting this into the above inequality and using the bound on the size of $\delta$, we find that 
\[
|\phi(t_1)+\phi(t_2)-\phi(t_1')-\phi(t_1+t_2-t_1')| 
\leq M_{\phi} C_{\phi} R^{-k}
\]
for some constant $M_{\phi}$ depending on $\phi$ and $\lf$. 
Then if we can understand the inequality
\[
|\phi(t_1+t_2-t_1')-\phi(t_1)-\phi(t_2)+\phi(t_1')| 
\leq m R^{-k} 
\]
for some $m>0$ to be defined later, we can then define $C_{\phi}$ by setting $C_{\phi}=mM_{\phi}^{-1}$.

Using the S.O.D.ing polynomial $\psi$ of $\phi$, we find that the above is 
\[
|(t_1-t_1')(t_2-t_1')\psi(t_1,t_2,t_1')| 
\leq m R^{-k}
.\] 
From our assumption that $\dist(I_1,I_1') \geq R^{-1}$, we have that $|t_1-t_1'| \geq R^{-1}$ and deduce that 
\begin{equation}\label{eq:somename}
|(t_2-t_1')\psi(t_1,t_2,t_1')| 
\leq m R^{1-k}
.\end{equation} 
Suppose first that $|t_2-t_1'| \leq R^{-1}$. In this situation, the linear inequality in \eqref{close_points:KdV} implies that $|t_1-t_2'| \leq 2R^{-1}$, and it follows that there are at most $5^4=625$ possible choices for $\{I_2,I_2'\}$ given that $\{I_1,I_1'\}$ is fixed. 
We briefly explain this. The worst case arises when $K=\C$ upon taking into account $I_1$ itself together with its 8 possible neighbors as well as their neighbors which gives 25 possibilities for $I_1$. Similarly,  there are at most 25 independent possibilities for $I_1'$. When $K=\R$, each interval has at most two neighbors yielding at most 5 choices for $I_1$ and 5 independent choices for $I_1'$. In the non-Archimedean case, the ultrametric inequality implies that no neighbors have to be taken into account. Therefore, the bound of 625 suffices for all local fields. 

Having considered the situation where $|t_2-t_1'| \leq R^{-1}$, assume that $|t_2-t_1'| \geq R^{-1}$. Then \eqref{eq:somename} implies that $|\psi(t_1,t_2,t_1')| \leq m R^{2-k}$. 
For each admissible pair $(t_1,t_1')$, the nontrivial, univariate polynomial equation 
\(
\psi(t_1,Y,t_1') = 0
\) 
has at most $\deg(\psi) = \deg(\phi)-2 = k-2$ possible roots (not necessarily distinct) in $\lf$.  
Label these roots $y_1=y_1(t_1,t_1'),\dots,y_d=y_d(t_1,t_1')$ where $d=d(t_1,t_1') \leq \deg(\psi)$. Then $\psi(t_1,Y,t_1')$ is of the shape 
\begin{equation}\label{def:splitting}
\psi(t_1,Y,t_1') = P_{t_1,t_1'}(Y)\prod_{i=1}^d (Y-y_i),
\end{equation}
where $P_{t_1,t_1'}(Y) \in \lf[Y]$ does not have any roots in $\lf$. 

Since the polynomial $P_{t_1,t_1'}$ has no roots and $\roi$ is compact, $|P_{t_1,t_1'}(t_2)|$ has a positive minimum over $t_2 \in \roi$; call this minimum $m({t_1,t_1'})$. At this point we may finally define $m := \inf_{t_1,t_1' \in \roi} m({t_1,t_1'})$; this is the aforementioned factor of $C_{\phi}$. Consequently, the inequality $|\psi(t_1,t_2,t_1')| \leq m R^{2-k}$ directly implies that $|\prod_{i=1}^d (t_2-y_i)| \leq R^{2-k}$. 
Thus, there exists an $i$ with $1 \leq i \leq d$ such that $|t_2-y_i| \leq R^{-1}$. 

For each $i \in \{1, \ldots, d(t_1,t_1')\}$, let $J_i=J_i(t_1,t_1')$ be the $R^{-1}$-neighborhood in $\partition_{R^{-1}}$ containing the point $y_i=y_i(t_1,t_1')$. The intervals $\{J_1,\dots,J_d\}$ and their neighbors contribute to the possible $I_2$; there are at most $9d$ possible intervals. Fixing one such possible interval $I_2$, the interval $I_2'$ is determined by the inequality $|t_1+t_2-t_1'-t_2'| \leq R^{-k}$. In fact, there are 9 possible intervals, one determined by the equation $t_1+t_2-t_1'-t_2'=0$ and possibly eight more from its neighbors. This gives at most $81d \leq 625(\deg{\phi}-2)$ possible pairs $(I_2,I_2')$ of intervals in this case. Upon combining this estimate with the bound of 625 possible choices found above we recover the bound postulated in the statement of the proposition. 
\end{proof}

% % % 
\section{Proof of Theorem~\ref{theorem:main:2D}}\label{section:CauchySchwarz}
% % % 

Suppose that $B \subset \lf^2$ is a square of diameter $(C_{\phi})^{-1}R^k$. 
Without loss of generality we may assume that $B$ contains the origin. Also, since the statement of Proposition \ref{prop:KdV} for any specific value of $C_{\phi}$ implies the same statement for all $C'<C_\phi$, we can assume without loss of generality that $C_{\phi} \leq 1$, which we henceforth do. 
By linearity of integration, we write 
\[
E_{\roi} = \sum_{I \in \partition_{R^{-1}}(\roi)} E_I
\]
and open up the $L^4$-norm 
\begin{align}\label{S1+S2}
\int |E_{\roi} f|^4 W_B
&= 
\sum_{I_1,I_2,I_1',I_2' \in \partition_{R^{-1}}(\roi)} \int E_{I_1}f E_{I_2}f \overline{E_{I_1'}f E_{I_2'}f} W_B \nonumber
\\ &=
\sum_{ \substack{I_1,I_1' \in \partition_{R^{-1}}(\roi) \\ \dist(I_1,I_1') < R^{-1}} } \sum_{I_2,I_2' \in \partition_{R^{-1}}(\roi)} \int E_{I_1}f E_{I_2}f \overline{E_{I_1'}f E_{I_2'}f} W_B \nonumber
\\ & \qquad\qquad +
\sum_{ \substack{I_1,I_1' \in \partition_{R^{-1}}(\roi) \\ \dist(I_1,I_1') \geq R^{-1}} } \sum_{I_2,I_2' \in \partition_{R^{-1}}(\roi)} \int E_{I_1}f E_{I_2}f \overline{E_{I_1'}f E_{I_2'}f} W_B \nonumber
\\ & =: T_1 + T_2
.\end{align}

Integration over the square $B$ removes those quadruples of intervals $(I_1,I_2,I_1',I_2')$ which fail to satisfy the conclusion of Proposition~\ref{prop:KdV}; in other words, all of those quadruples contribute nothing to the integral. We take a moment to state this fact precisely and prove it. 
\begin{lemma}\label{lemma:orthogonality}
Suppose that $I_1,I_2, I_1',I_2'$ are in $\partition_{R^{-1}}(\roi)$. If \eqref{close_points:KdV} is not satisfied, then 
\begin{equation}\label{eq:orthogonality}
\int E_{I_1}f E_{I_2}f \overline{E_{I_1'}f E_{I_2'}f} \; W_B 
= 
0
.\end{equation}
\end{lemma}

\begin{proof}
We write out the integral and use Fourier inversion, where, for simplicity, we abbreviate $\mathcal I = I_1 \times I_2 \times I_1' \times I_2'$ and  $\dd{\vof{\xi}} =\dd{\xi_1}\dd{\xi_2}\dd{\xi_1'}\dd{\xi_2'}$. Thus, we deduce that
\begin{align*}
\int_{\lf^2} & E_{I_1}f E_{I_2}f \overline{E_{I_1'}f E_{I_2'}f} W_B 
\\&= 
\int_{\lf^2} \int_{\mathcal I} f(\xi_1)f(\xi_2) \overline{f(\xi_1')}\overline{f(\xi_2')} \eof{[\gamma(\xi_1)+\gamma(\xi_2)-\gamma(\xi_1')-\gamma(\xi_2')]\cdot\vof{x}} W_B(\vof{x}) \dd{\vof{\xi}}\dd{\vof{x}}
\\&= 
\int_{\mathcal I} f(\xi_1)f(\xi_2) \overline{f(\xi_1')}\overline{f(\xi_2')} 
\int_{\lf^2} \eof{[\gamma(\xi_1)+\gamma(\xi_2)-\gamma(\xi_1')-\gamma(\xi_2')]\cdot\vof{x}} W_B(\vof{x})\dd{\vof{x}} \dd{\vof{\xi}}
\\&= 
\int_{\mathcal I} f(\xi_1)f(\xi_2) \overline{f(\xi_1')}\overline{f(\xi_2')} 
\widehat{W_B}(\gamma(\xi_1)+\gamma(\xi_2)-\gamma(\xi_1')-\gamma(\xi_2')) \dd{\vof{\xi}}
.\end{align*}
Observe that $\widehat{W_B}$ is supported on the square of diameter $C_\phi R^{-k}$ centered at the origin. Therefore, $\widehat{W_B}$ is 0 whenever \eqref{close_points:KdV} is not satisfied, and so are the outer integrals. 
\end{proof}

For each $I_1$, there are at most 9 intervals $I_1'$ such that $\dist(I_1,I_1') < R^{-1}$. Moreover, for each choice of $I_1$, $I_1'$ and $I_2$, it follows from Lemma~\ref{lemma:orthogonality} and \eqref{close_points:KdV} that at most 9 intervals $I_2'$ make a positive contribution to the integral. Applying the Cauchy--Schwarz inequality first over the sum in $(I_1,I_1')$ and then over $(I_2,I_2')$, we obtain 
\begin{equation}\label{S1-bound}
T_1 \leq 9\int \sum_{I_1 \in \partition_{R^{-1}}(\roi)} |E_{I_1}f|^2\sum_{I_2,I_2' \in \partition_{R^{-1}}(\roi)}  E_{I_2}f \overline{E_{I_2'}f} W_B
\leq 81 \int |S_{\partition_{R^{-1}}(\roi)}f|^4 W_B
.\end{equation}
With $T_1$ in hand, we turn our attention to $T_2$. 

For a pair of intervals $(I_1, I_1')$ denote by $\mathcal{S}(I_1,I_1')$ the set of all pairs $(I_2,I_2')$ satisfying the conclusion of Proposition~\ref{prop:KdV}. 
Upon deploying Lemma~\ref{lemma:orthogonality} within the expression for $T_2$ and applying the Cauchy--Schwarz inequality we deduce that 
\begin{align}\label{S2-bound}
T_2
& = 
\int W_B \sum_{ \substack{I_1,I_1' \in \partition_{R^{-1}}(\roi) \\ \dist(I_1,I_1') \geq R^{-1}} } E_{I_1}f \overline{E_{I_1'}f} \sum_{(I_2,I_2') \in \mathcal{S}(I_1,I_1')} E_{I_2}f \overline{E_{I_2'}f}  \nonumber
\\ & \leq 
\int W_B \bigg( \sum_{ \substack{I_1,I_1' \in \partition_{R^{-1}}(\roi) \\ \dist(I_1,I_1') \geq R^{-1}} } |E_{I_1}f E_{I_1'}f|^2 \bigg)^{1/2} \bigg( \sum_{ \substack{I_1,I_1' \in \partition_{R^{-1}}(\roi) \\ \dist(I_1,I_1') \geq R^{-1}} }  \big[ \sum_{(I_2,I_2') \in \mathcal{S}(I_1,I_1')} |E_{I_2}f E_{I_2'}f| \big]^2 \bigg)^{1/2} \nonumber
\\ & \leq 
\int W_B (S_{\partition_{R^{-1}}(\roi)} f)^{2} \bigg( \sum_{ \substack{I_1,I_1' \in \partition_{R^{-1}}(\roi) \\ \dist(I_1,I_1') \geq R^{-1}} } \big[ \sum_{(I_2,I_2') \in \mathcal{S}(I_1,I_1')} |E_{I_2}f E_{I_2'}f| \big]^2 \bigg)^{1/2} 
.\end{align}

We now analyze the expression 
\begin{equation}\label{Upsilon}
\Upsilon = \sum_{ \substack{I_1,I_1' \in \partition_{R^{-1}}(\roi) \\ \dist(I_1,I_1') \geq R^{-1}} } \big[ \sum_{(I_2,I_2') \in \mathcal{S}(I_1,I_1')} |E_{I_2}f E_{I_2'}f| \big]^2.
\end{equation}
The Cauchy--Schwarz inequality applied to the inner sum implies that 
\begin{align*}
\bigg( \sum_{(I_2,I_2') \in \mathcal{S}(I_1,I_1')} |E_{I_2}f E_{I_2'}f| \bigg)^2 
\leq 
|\mathcal{S}(I_1,I_1')| \hspace{-0em} \sum_{(I_2,I_2') \in \mathcal{S}(I_1,I_1')} |E_{I_2}f E_{I_2'}f|^2. 
\end{align*}
From Proposition~\ref{prop:KdV}, we see that uniformly over all $R \in \mathcal R(K)$ and pairs $(I_1,I_1') \in \partition_{R^{-1}}(\roi) \times \partition_{R^{-1}}(\roi)$ such that $\dist(I_1,I_1') \geq R^{-1}$, we have  $|\mathcal{S}(I_1,I_1')| \leq 625(\deg{\phi}-1)$. Thus, we conclude that 
\[
\Upsilon \leq 625(\deg{\phi}-1)\sum_{I_1,I_1' \in \partition_{R^{-1}}(\roi)} \sum_{(I_2,I_2') \in \mathcal{S}(I_1,I_1')} |E_{I_2}f E_{I_2'}f|^2.
\]

At the same time, it is apparent from the definition of the sets $\mathcal S(I_1, I_1')$ together with the symmetry of the underlying inequalities \eqref{close_points:KdV} that $(I_2, I_2') \in \mathcal S(I_1, I_1')$ for a pair of intervals $(I_1, I_1')$ if and only if $(I_1, I_1') \in \mathcal S(I_2, I_2')$.  Using this information to invert the order of summation in the remaining double sum and applying Proposition~\ref{prop:KdV}, we find that 
\begin{align*}
\sum_{I_1,I_1' \in \partition_{R^{-1}}(\roi)} \sum_{(I_2,I_2') \in \mathcal{S}(I_1,I_1')} |E_{I_2}f E_{I_2'}f|^2 
&= 
\sum_{I_2,I_2' \in \partition_{R^{-1}}(\roi)} |E_{I_2}f E_{I_2'}f|^2  |\mathcal{S}(I_2,I_2')| \\
&\leq 
625(\deg{\phi}-1)\sum_{I_2,I_2' \in \partition_{R^{-1}}(\roi)} |E_{I_2}f E_{I_2'}f|^2  \\
&= 
625(\deg{\phi}-1) (S_{\partition_{R^{-1}}(\roi)} f)^4
.\end{align*}
Collecting our estimates, we find that 
\[
\Upsilon 
\leq 
[625(\deg{\phi}-1)]^2 (S_{\partition_{R^{-1}}(\roi)} f)^4,
\]
and consequently upon recalling \eqref{S2-bound} and the definition \eqref{Upsilon} we discern that
\[
T_2 \leq \int W_B (S_{\partition_{R^{-1}}(\roi)} f)^{2} \Upsilon^{1/2} 
\leq 
625(\deg{\phi}-1) \int W_B (S_{\partition_{R^{-1}}(\roi)} f)^{4}
.\]
It remains to combine this last bound with \eqref{S1+S2} and \eqref{S1-bound}. Thus, we obtain the bound 
\[
\int W_B |E_{\roi} f|^4 
\leq 
625\deg{\phi} \int W_B (S_{\partition_{R^{-1}}(\roi)} f)^{4},
\]
with an implicit constant depending only on $\phi$ and the field $K$, and the proof of the theorem is complete upon raising both sides to the power $1/4$.

% % %
\section{The Rolle of curvature I: Archimedean fields}\label{section:real}
% % %

In this section we give a variant of our previous argument to prove Theorem~\ref{theorem:real:2D}, thereby extending the polynomial condition of Theorem~\ref{theorem:main:2D} to a finite-type curvature condition over $\R$. 
The main ingredient in this extension, as well as a second one concerning functions $\phi$ which are differentiable and strictly convex, is a version of Rolle's Theorem. 
While Rolle's Theorem is primarily a result in real analysis, there is a sufficient surrogate for it over the complex numbers given by the Voorhoeve index, see \cite{Voorhoeve,KY}. Thus, we are able to extend Theorem~\ref{theorem:main:2D} even to complex analytic functions $\phi$, as long as they are sufficiently close to a perturbation of a polynomial. We treat each of these extensions in turn.

% % 
\subsection{The proof of Theorem~\ref{theorem:real:2D}}
% % 

In this subsection and the next, we assume that our local field is $\R$. {Recall that this means that $\roi:=[-1/2,1/2]$. In this section define the sum-set $\roi+\roi-\roi=3\roi$. Observe that $3\roi$ contains $\roi$. } 
We extend Proposition~\ref{prop:KdV} to finite-type, degenerate curves in $\R^2$. 
\begin{proposition}\label{prop:KdV:real}
Let $\mydegree \geq 2$. 
Suppose that $\phi \in C^{\mydegree}(3\roi)$ such that $\phi^{(k)}(t) \neq 0$ for all $t$ in the interval $3\roi$, and that $R \in \scales(\R)$. 
Fix $I_1,I_1'$ in $\partition_{R^{-1}}$ such that $\dist(I_1,I_1') \geq R^{-1}$. 
There exists a positive constant $C_{\phi}$, depending only on $\phi$, such that there are at most $625(k-1)$ pairs $(I_2,I_2')$ of intervals in $\partition_{R^{-1}}$ with the property that for all points $t_1 \in I_1$, $t_2 \in I_2$, $t_1' \in I_1'$ and $t_2' \in I_2'$, we have 
\begin{equation}\label{close_points:KdV:real}
|\gamma(t_1)+\gamma(t_2)-\gamma(t_1')-\gamma(t_2')| 
\leq C_{\phi} R^{-k}
.\end{equation}
\end{proposition}

When $\mydegree=2$, the condition $\phi \in C^{\mydegree}(\R)$ such that $|\phi^{(\mydegree)}(t)|>0$ for all $t \in [0,1]$ is the `non-degeneracy' condition used in \cite{GGPRY} and many other works in harmonic analysis. 
Admittedly, our condition is on a bigger set than considered in \cite{GGPRY}. However, pigeonholing and rescaling allow us to replace it with $\roi=[-1/2,1/2]$ at the cost of inflating the constants by a harmless, uniform factor. 

Once Proposition~\ref{prop:KdV:real} is proved, the deduction of Theorem~\ref{theorem:real:2D} follows identically as in Section~\ref{section:CauchySchwarz}. 
The proof of Proposition~\ref{prop:KdV:real} also relies on Wooley's S.O.D.ing method, this time adapted to smooth functions rather than polynomials. If $\phi$ is a $C^{k}$ function for some $k \geq 2$, then analogous to before we define $\psi(X,Y,Z)$ as in \eqref{def:SOD} to be the trivariate second order differencing function associated with $\phi$. The function $\psi$ is not necessarily a polynomial; however, it is contained in the set $C^{k-2}([3\roi]^3)$. 
\begin{lemma}\label{lemma:non-deg_implies_definite_SOD}
Let $\mydegree \geq 2$ and $\phi \in C^{\mydegree}(3\roi)$ with trivariate S.O.D.ing function $\psi(X,Y,Z)$. Suppose that $\phi$ satisfies $\phi^{(\mydegree)}(t) \neq 0$ for all $t \in 3\roi$. If $t_1, t_1' \in \roi$ with $t_1 \neq t_1'$, then the univariate S.O.D.ing function $\psi(t_1,Y,t_1')$ has at most $\mydegree-2$ distinct zeros in $\roi\setminus\{t_1'\}$. 
\end{lemma}
\noindent 
Once this lemma is proved, the proof of Proposition~\ref{prop:KdV:real} is identical to that of Proposition~\ref{prop:KdV}. We leave the deduction of Proposition~\ref{prop:KdV:real} to the reader, but remind the reader that \eqref{approx:Taylor} continues to hold for functions $\phi \in C^{\mydegree}$ by Taylor's Theorem.

Lemma~\ref{lemma:non-deg_implies_definite_SOD} reveals that the finite type (quantification of non-degeneracy) of the curve $\gamma(X):=(X,\phi(X))$ is reflected in the geometry of the S.O.D.ing function when $\lf=\R$. 
As an example, suppose that $\mydegree=2$ and $\phi''$ has no zero on $3\roi$. Then $\phi''$ is either always positive or always negative on $3\roi$. By replacing $\phi$ by $-\phi$ if necessary, we assume that $\phi''>0$ so that the curve $\gamma(X)$ is strictly convex on $3\roi$. Lemma~\ref{lemma:non-deg_implies_definite_SOD} informs us that the strict convexity of this curve implies that there are no non-trivial zeros to $\psi_\phi(X,Y,Z)$. 
In conjunction with our method above, this gives another alternate proof of the two-dimensional case of the main result of \cite{GGPRY}. 

In the proof of Lemma~\ref{lemma:non-deg_implies_definite_SOD}, we need the following classical, iterated version of Rolle's Theorem; it is easily proven by induction using the better known version for $k=2$ as the base case. 
\begin{Rolle}[Iterated version]
Suppose that $I$ is a closed interval in $\R$ and $\mydegree \in \N$ with $\mydegree \geq 2$. 
If $f \in C^{\mydegree}(I)$ and $a_1<a_2<\dots<a_\mydegree$ are points in $I$ such that $f(a_1)=\dots=f(a_\mydegree)$, then there exists a point $\zeta \in I$ such that $f^{(\mydegree-1)}(\zeta)=0$. 
\end{Rolle}

\begin{proof}[Proof of Lemma~\ref{lemma:non-deg_implies_definite_SOD}]
Fix $t_1,t_1' \in \roi$ with $t_1 \neq t_1'$. 
If $Y \neq t_1'$, then it is clear from \eqref{def:SOD} that the zeros of $\psi(t_1,Y,t_1')$ correspond to the remaining zeros of 
$f(Y) := \phi(t_1+Y-t_1')-\phi(t_1)-\phi(Y)+\phi(t_1')$. 
For the sake of contradiction, suppose that $\psi(t_1,Y,t_1')$ has $\mydegree-1$ distinct zeros, then $f$ is a $C^\mydegree(\roi)$-function with $\mydegree$ distinct zeros; the $\mydegree^{\th}$-zero being $Y=t_1'$. The iterated version of Rolle's Theorem implies that there exists a point $\zeta \in \roi$ such that 
\[
0 = f^{(\mydegree-1)}(\zeta) 
= 
\phi^{(\mydegree-1)}(t_1+\zeta-t_1') -\phi^{(\mydegree-1)}(\zeta)
.\]
Consequently, $\phi^{(\mydegree-1)}$ has the same value at two distinct points, and therefore, an application of the classical form of Rolle's Theorem implies that $\phi^{(\mydegree)}$ is 0 at some point in $3\roi$; this is a contradiction. 
\end{proof}

% % 
\subsection{Weakening the regularity assumption}
% % 

In this subsection, we assume that our local field is $\R$. 
A slight modification of the proof of Lemma~\ref{lemma:non-deg_implies_definite_SOD} when $\mydegree=2$ allows us to weaken the regularity requirements of Theorem~\ref{theorem:real:2D} to the following. 
\begin{theorem}\label{theorem:convex}
Suppose that $\phi$ is differentiable everywhere in $3\roi$ and is strictly convex or strictly concave on $3\roi$. Define $\gamma(T)$ to be the planar curve $(T,\phi(T))$. 
There exists a positive constant $C_{\phi}$, depending only on $\phi$, such that for all $R \in \N$ and all balls $B$ of diameter at least $C_{\phi} R^2$ in $\R^2$, we have the inequality 
\begin{equation*}
\| E_{\roi} f \|_{L^4(B)} 
\lesssim_{\gamma} \; 
\| S_{\partition_{R^{-1}}(\roi)} f \|_{L^4(B)} 
.\end{equation*}
\end{theorem}
\noindent The condition of strict convexity can be seen to be necessary by taking $\phi(X)=X$, in which case the theorem fails. In particular, $\phi$ is convex and differentiable on $\roi$ (meaning the left and right limits coincide at each point); this implies that $\phi$ is \emph{continuously} differentiable on $\roi$.

The proof of Theorem~\ref{theorem:convex} is almost identical to that of Theorem~\ref{theorem:main:2D} once the following surrogate for Propositions~\ref{prop:KdV:real} or \ref{prop:KdV} is proved. 
\begin{proposition}\label{prop:KdV:convex}
Suppose that $\phi$ is differentiable everywhere in $3\roi$ and is strictly convex or strictly concave on $3\roi$, and that $R \in \scales(\R)$. 
Fix $I_1,I_1'$ in $\partition_{R^{-1}}$ such that $\dist(I_1,I_1') \geq R^{-1}$. 
There exists a positive constant $C_{\phi}$, depending only on $\phi$, such that there are at most $5^4=625$ pairs $\{I_2,I_2'\}$ of intervals in $\partition_{R^{-1}}$ with the property that for all points $t_1 \in I_1$, $t_2 \in I_2$, $t_1' \in I_1'$ and $t_2' \in I_2'$, we have 
\begin{equation}\label{close_points:KdV:convex}
|\gamma(t_1)+\gamma(t_2)-\gamma(t_1')-\gamma(t_2')| 
\leq C_{\phi} R^{-2}
.\end{equation}
\end{proposition}

In turn, Proposition~\ref{prop:KdV:convex} is proved similarly to Propositions~\ref{prop:KdV} and \ref{prop:KdV:real} once we have established the following proxy for Lemma~\ref{lemma:non-deg_implies_definite_SOD}.

\begin{lemma}\label{lemma:convexity_implies_definite_SOD}
Suppose that $\phi$ is differentiable everywhere in $3\roi$ and is strictly convex or strictly concave on $3\roi$ with trivariate S.O.D.ing function $\psi(X,Y,Z)$. If $t_1, t_1' \in \roi$ with $t_1 \neq t_1'$, then the univariate S.O.D.ing function $\psi(t_1,Y,t_1')$ has no zeros on $\roi\setminus\{t_1'\}$. 
\end{lemma}

The deduction of Proposition~\ref{prop:KdV:convex} differs from that of Proposition~\ref{prop:KdV:real} only in one detail, namely our ability to replace \eqref{close_points:KdV:convex} by the inequality 
\begin{equation}\label{estimate:}
|\phi(t_1)+\phi(t_2)-\phi(t_1')-\phi(t_1+t_2-t_1')| 
\lesssim_\phi R^{-2}
.\end{equation}
After replacing $\phi$ by $-\phi$ if necessary, we may suppose that $\phi$ is a strictly convex function, so that in particular $\phi'$ is strictly increasing. 

The linear equation in \eqref{close_points:KdV} implies that for $\delta:=t_1+t_2-t_1'-t_2'$, we have $|\delta| \lesssim_\phi R^{-2}$. 
Since $\phi \in C^1(3\roi)$ and $3\roi$ is compact, Taylor's Theorem implies that 
\begin{equation}\label{approx:C1}
|\phi(t_1+t_2-t_1'-\delta) - \phi(t_1+t_2-t_1')| 
\lesssim |\delta| 
.\end{equation}
The implicit constant is uniform in all relevant parameters. 
The estimate \eqref{approx:C1} together with the bound on $\delta$ and the triangle inequality imply that 
\[
|\phi(t_1)+\phi(t_2)-\phi(t_1')-\phi(t_1+t_2-t_1')| 
\lesssim_\phi R^{-2}
.\]
The remaining analysis following this inequality is identical to that in the proofs of Proposition~\ref{prop:KdV:real} and \ref{prop:KdV} upon replacing Lemma~\ref{lemma:non-deg_implies_definite_SOD} by Lemma~\ref{lemma:convexity_implies_definite_SOD}, so it suffices to prove Lemma~\ref{lemma:convexity_implies_definite_SOD}.

\begin{proof}[Proof of Lemma~\ref{lemma:convexity_implies_definite_SOD}]
Fix $t_1,t_1' \in \roi$ with $t_1 \neq t_1'$. 
If $Y \neq t_1'$, then the zeros of $\psi(t_1,Y,t_1')$ correspond to the remaining zeros of 
$f(Y) := \phi(t_1+Y-t_1')-\phi(t_1)-\phi(Y)+\phi(t_1')$. 
For the sake of contradiction, suppose that $\psi(t_1,Y,t_1')$ has a zero, then $f(Y)$ is a $C^1(\roi)$-function with two distinct zeros; the second zero being $Y=t_1'$. Rolle's Theorem implies that there exists a point $\zeta \in \roi$ such that 
\[
0 = f'(\zeta) 
= 
\phi'(t_1+\zeta-t_1')-\phi'(\zeta)
.\]
Since $t_1 \neq t_1'$, this is impossible by strict convexity. 
\end{proof}

% % 
\subsection{A complex version}
% % 

In this subsection, we assume that our local field is $\C$, and recall our convention that for a number $z=x+iy \in \C$ we write $|z|=|z|_\C=\max\{|x|_\R,|y|_\R\}$, so that the set $\roi=\{z \in \C: |z|_{\C} \leq 1/2\}$ has the shape of a square. {Once again, we will make use of the sum and difference set $\roi+\roi-\roi=3\roi$.} 

Based on Rouche's Theorem from complex analysis, one might expect that Theorem~\ref{theorem:main:2D} over $\C$ may extend to analytic functions which are perturbations of polynomials in the domain $\roi$. The putative difficulty in making this rigorous is relating the zeros of an analytic function to those of its S.O.D.ing function. 
Fortunately, there is a sufficient surrogate for Rolle's Theorem over $\C$ given by the Voorhoeve index introduced in \cite{Voorhoeve}. This gives us the following theorem. 
\begin{theorem}\label{theorem:complex}
Let $\mydegree \geq 2$, $\beta>0$ and $\phi$ be an analytic function on the square $3\roi \subset \C$ such that $\phi^{(\mydegree-1)}$ has no zeros on the boundary of the square $\roi$, $|\phi^{(\mydegree)}| \geq \beta$ on $\roi$ and $|\phi^{(\mydegree+1)}(t)| < \frac{\beta}{2\sqrt{2}}$ for all $t \in 3\roi$. Define $\gamma(T)$ to be the planar curve $(T,\phi(T))$. 
There exists a positive constant $C_{\phi}$, depending only on $\phi$, such that for all $R \in \N$ and all balls $B$ of width at least $C_{\phi} R^2$ in $\C^2$, we have the inequality 
\begin{equation*}
\| E_{\roi} f \|_{L^4(B)} 
\lesssim_{\gamma} \; 
\| S_{\partition_{R^{-1}}(\roi)} f \|_{L^4(B)} 
.\end{equation*}
\end{theorem}

The proof of this theorem follows the same structure as in the proof of Theorem~\ref{theorem:real:2D} upon replacing Lemma~\ref{lemma:non-deg_implies_definite_SOD} by the following lemma. 
\begin{lemma}\label{lemma:non-deg_implies_definite_SOD:complex}
Let $\mydegree \geq 2$, $\phi$ be an analytic function on the square $\roi$ satisfying the hypotheses of Theorem~\ref{theorem:complex} and $\psi(X,Y,Z)$ be the trivariate S.O.D.ing function of $\phi$. 
If $t_1, t_1' \in \roi$ with $t_1 \neq t_1'$, then the univariate S.O.D.ing function $\psi(t_1,Y,t_1')$ has at most $\mydegree-2$ distinct zeros on $\roi\setminus\{t_1'\}$. 
\end{lemma}

\begin{proof}
In this proof, let $\|\cdot\|$ denote the Euclidean norm on $\C$. This is different to the norm $|\cdot|_\C$ used previously and below. 
For meromorphic functions $f$ defined in an open neighborhood of $\roi$, the Voorhoeve Index of $f$ is defined as 
\[
V(f) 
:= 
\frac{1}{2\pi} \oint \big| \partial_t \Arg f(t) \big| \dd{t}
= 
\frac{1}{2\pi} \oint \big| \ip\bigg(\frac{f'(t)}{f(t)} \bigg) \big| \dd{t},
\]
where the notation $\oint$ indicates that the integral runs over the boundary of the square $\roi$. 
Fix $t_1,t_1' \in \roi$ with $t_1 \neq t_1'$ and define $f(Y) := \phi(t_1+Y-t_1')-\phi(t_1)-\phi(Y)+\phi(t_1')$. 
We claim that 
\begin{equation}\label{estimate:VI}
V(f^{(\mydegree-1)}) < 1. 
\end{equation}
Assuming that this is true, it follows from Corollary~2 on page~107 of \cite{KY} that $V(f) < \mydegree$. The Argument Principle then implies that the number of zeros of $f$ is at most $\mydegree-1$, and hence that $\psi(t_1,Y,t_1')$ has at most $\mydegree-2$ distinct zeros. See \cite{Voorhoeve,KY} for a discussion of how the Argument Principle applies. 

It remains to prove \eqref{estimate:VI}. From the fundamental theorem of calculus, we have that 
\begin{align*}
f^{(\mydegree-1)}(Y) 
& = 
\phi^{(\mydegree-1)}(t_1+Y-t_1')-\phi^{(\mydegree-1)}(Y)
= 
\int_{\mathcal{C}} \phi^{(\mydegree)}(Y+s) \dd{s}
\end{align*}
for any simple curve $\mathcal{C} \subset 2\roi$ running from $0$ to $t_1-t_1'$. We fix such a curve and denote its length by $\ell$. 
Estimating the imaginary part trivially and inserting the above equation, we find 
\begin{align}\label{eq:Voorhoeve}
V(f^{(\mydegree-1)}) 
& = 
\frac{1}{2\pi} \oint \left| \ip\bigg(\frac{f^{(\mydegree)}(Y)}{f^{(\mydegree-1)}(Y)} \bigg) \right| \dd{Y} 
\leq 
\frac{1}{2\pi} \oint \bigg\| \frac{f^{(\mydegree)}(Y)}{f^{(\mydegree-1)}(Y)} \bigg\| \dd{Y} \nonumber
\\&= 
\frac{1}{2\pi} \oint \bigg\|  \frac{\phi^{(\mydegree)}(t_1+Y-t_1')-\phi^{(\mydegree)}(Y)}{\int_{\mathcal{C}} \phi^{(\mydegree)}(Y+s)\dd{s}}  \bigg\| \dd{Y}.
\end{align}
By Taylor's theorem, we have $\phi^{(\mydegree)}(Y+s)=\phi^{(\mydegree)}(Y)+\rho_Y(s)$ with some remainder $\rho_Y(s)$ satisfying 
\[
\|\rho_Y(s)\| 
\leq 
\sup_{x \in 2\roi} \|\phi^{(k+1)}(x)\| \|s\| \leq \sqrt 2\sup_{x \in 2\roi} \|\phi^{(k+1)}(x)\|
\] 
for all $s \in 2\roi$. Thus, we can bound the denominator in \eqref{eq:Voorhoeve} from below via
\begin{align}\label{eq:Voorhoeve_denominator}
\left\| \int_{\mathcal{C}} \phi^{(\mydegree)}(Y+s)\dd{s} \right\| 
&= 
\left\| \int_{\mathcal{C}} \phi^{(\mydegree)}(Y)+ \rho_Y(s)\dd{s} \right\| \nonumber
\\& \geq 
\left\| \int_{\mathcal{C}} \phi^{(\mydegree)}(Y) \dd{s} \right\| - \left\| \int_{\mathcal{C}} \rho_Y(s)\dd{s} \right\| \nonumber
\\&\geq 
\ell \left( \|\phi^{(\mydegree)}(Y)\| - \sqrt 2\sup_{x \in 2\roi} \|\phi^{(k+1)}(x)\|\right) \nonumber
\\&\geq \frac12  \ell  \| \phi^{(\mydegree)}(Y)\|
\end{align}
whenever 
\begin{equation}\label{eq:sup_vs_inf}
\sqrt 2\sup_{x \in 2\roi} \|\phi^{(k+1)}(x)\| 
\leq 
\frac{1}{2} \inf_{|Y|_{\C}=1/2}\| \phi^{(\mydegree)}(Y)\|.
\end{equation}
This last inequality follows from the hypotheses of the lemma.

Meanwhile, using the fundamental theorem of calculus as above, we can estimate the numerator in \eqref{eq:Voorhoeve} as 
\begin{align}\label{eq:Voorhoeve_numerator}
\|\phi^{(\mydegree)}(t_1+Y-t_1')-\phi^{(\mydegree)}(Y)\| 
&= 
\left\|\int_{\mathcal{C}} \phi^{(\mydegree+1)}(Y+s) \dd{s}\right\| 
\leq 
\ell\sup_{s \in \mathcal{C}}\|\phi^{(\mydegree+1)}(Y+s)\| 
.\end{align}
Now, upon inserting both \eqref{eq:Voorhoeve_denominator} and \eqref{eq:Voorhoeve_numerator} into \eqref{eq:Voorhoeve}, and noting that the perimeter of the square $\roi$ has length $4$, we find
\begin{align*}
V(f^{(k-1)}) 
&\leq 
\frac{1}{2\pi} \oint \frac{\ell\sup_{s \in \mathcal{C}} \|\phi^{(\mydegree+1)}(s)\| }{\frac12  \ell  \| \phi^{(\mydegree)}(Y)\|} \dd{Y}
\leq 
\frac{4}{\pi}\frac{\sup_{s \in 3\roi} \|\phi^{(\mydegree+1)}(s)\|}{\inf_{|Y|_{\C}=1/2}\| \phi^{(\mydegree)}(Y)\|}
\leq 
\frac{\sqrt{2}}{\pi}
,\end{align*}
where in the last step we used the bound \eqref{eq:sup_vs_inf}.
\end{proof}

% % % 
\section{The Rolle of curvature II: Non-Archimedean fields}\label{section:Rolle}
% % % 

It would be nice to have a version of Theorem~\ref{theorem:real:2D} for $\oddprime$-adic curves when $k \geq 3$.  Our proof of Theorem~\ref{theorem:real:2D} relied on  Rolle's theorem. Rolle's theorem and related matters such as the Fundamental Theorem of Calculus, are inherently real phenomena exploiting the ordering of the real line. 
It is known that Rolle's theorem fails for non-Archimedean local fields, even for polynomials. In this section, we recall this failure over the fields $\Q_p$ and augment it to demonstrate that Lagrange interpolation with derivatives fails for non-Archimedean local fields. 
Going further, we compare the geometries of S.O.D.ing polynomials over $\R$ with those over non-Archimedean local fields. Recall that Proposition~\ref{prop:KdV:real} over $\R$ says that if $\phi \in C^2(\R)$ with $\phi'' \neq 0$ on $[-1,1]$, then its S.O.D.ing function has no non-trivial zeros on $\roi^3$. This is far from true over non-Archimedean fields where the S.O.D.ing polynomial may have many zeros near the origin despite $\phi'' \neq 0$ on all of $\roi$. These phenomena present technical obstructions in extending our methods to obtain a `smooth' $\oddprime$-adic version of Theorem~\ref{theorem:real:2D}. 
In our following arguments, we will focus on the case $K=\Q_p$ for some rational prime $p$.

% % 
\subsection{Lagrange interpolation for non-Archimedean local fields}
% % 

We commence with the following example from \cite{Katok} which shows the failure of Rolle's theorem for non-Archimedean fields. 
Fix a finite, rational prime $p$ and let $g(X):=X^p-X$. We have $g(0)=g(1)=0$, whereas the derivative is $g'(X)=p X^{p-1}-1$, so that $g'(x) \in -1+p \Z_p$ for all $x \in \Z_p$. Hence, $g'(x) \neq 0$ for all $x \in \Z_p$. 

Our next proposition builds on this observation to show that, while Lagrange interpolation holds for local fields (indeed, it holds for any field), the additional real phenomenon \eqref{eq:matching_derivatives} about matching derivatives fails for non-Archimedean local fields. 
To set the scene, we recall Lagrange interpolation for the reader's convenience. 
\begin{Lagrange}
Fix $k\in\N$ and $\lf$ a field. If $a_0,\dots,a_k \in \lf$ are distinct points and $b_0,\dots,b_k$ are values in $\lf$, then there exists a polynomial $P(X) \in \lf[X]$, depending on the points $a_0,\dots,a_k$ and values $b_0,\dots,b_k$, such that 
\[ P(a_i)=b_i \quad \text{for all} \quad 0 \leq i \leq k .\]
In fact, the polynomial 
\begin{equation}\label{LI}
P_{\vof{a},\vof{b}}(X) 
:= 
\sum_{m=0}^{k} b_m \bigg( \prod_{\substack{i=0\\i \neq m}}^k \frac{X-a_i}{a_m-a_i} \bigg) 
\end{equation}
is the unique such polynomial of degree at most $k$. 

Moreover, suppose that $\lf=\R$ and $\phi \in C^{k}(\roi)$. If $a_0,\dots,a_k \in \R$ are any distinct points and $P \in \R[X]$ is such that 
\( P(a_i)=\phi(a_i) \)
for all \( 0 \leq i \leq k \), 
then there exists a point $\zeta \in \roi$ such that 
\begin{equation}\label{eq:matching_derivatives} 
P^{(k)}(\zeta)=\phi^{(k)}(\zeta) 
.\end{equation}
\end{Lagrange}

The proof of the first part is obvious once \eqref{LI} is written down. The proof of \eqref{eq:matching_derivatives} follows from the iterated version of Rolle's Theorem in Section~\ref{section:real}. 

\begin{proposition}\label{prop:Rolle_fails}
Suppose that $p>2$ is a finite, rational prime. Define $f(X):=X^{p+1}-X^2$. If $q(X) \in \Q_p[X]$ is a quadratic polynomial such that $f(0)=q(0)$, $f(1)=q(1)$, $f(-1)=q(-1)$, then $q(X) \equiv 0$ is the 0-polynomial. 
In particular, there is no point $\zeta \in \rroi$ such that $f''(\zeta)=q''(\zeta)$. 
\end{proposition}

Observe that $|f''(t)| = 1$ for $t\in\Z_p$. So, the graph $(X,f(X))$ is a non-degenerate curve over $\Z_p$, and this proposition presents a technical obstruction in proving an analogue of Theorem~\ref{theorem:convex} for local fields.  

\begin{proof}[Proof of Proposition~\ref{prop:Rolle_fails}]
Let $f(X) = X^{p+1}-X^2$. Suppose we have a quadratic approximation $q(X)$ such that 
\[ 
q(0)=f(0)=q(1)=f(1)=q(-1)=f(-1)=0 .
\] 
Writing $q(X)=b_2X^2+b_1X+b_0$, the above may be written as the system of linear equations 
\[
\begin{pmatrix}
0 \\ 0 \\ 0 
\end{pmatrix}
= 
\begin{pmatrix}
1 & 0 & 0
\\
1 & 1 & 1 
\\ 
1 & -1 & 1 
\end{pmatrix}
\begin{pmatrix}
b_0 \\ b_1 \\ b_2 
\end{pmatrix}
.\]
The first equation tells us that $b_0=0$. 
Then the second equation tells us that $b_1+b_2=0$ while the last equation tells us that $-b_1+b_2=0$. Therefore, $b_1=b_2=0$ and $q(X)$ must be the 0-polynomial. 
The 0-polynomial has second derivative 0 everywhere, but $|f''(X)|=1$ for all $x \in \Z_p$ and $p>2$. 
\end{proof}

% % 
\subsection{Non-degenerate curves and the geometry of $p$-adic S.O.D.ing polynomials}
% % 

The analogue of Lemma~\ref{lemma:non-deg_implies_definite_SOD} fails for non-Archimedean fields. 
To warm-up, take $p=2$ and consider the curve $\gamma(X) := (X,\phi(X))$ where $\phi(X):=X^3-X^2/2$. 
This curve is non-degenerate on $\Z_2$ (but not on $2^{-1}\Z_2$) because $\phi''(x) \in -1+2\Z_2$ for all $x \in \Z_2$. Our second order difference polynomial for $\phi$ is 
\(
\psi(X,Y,Z) =
3(X+Y)-1
.\) 
Since $3$ is a unit in $\Z_2$, we see that this S.O.D.ing polynomial vanishes on a line in $\Z_2^2$. The following proposition extends this failure to higher degrees and other finite, rational primes. 

\begin{proposition}\label{prop:JB}
Let $k \geq 4$ be an integer and $p$ an odd, rational prime having the properties that $p \equiv 1 \mmod{k-2}$ and either $k=p+1$ or $k-1$ is not a $(k-2)$th power residue modulo $p$. Then the polynomial 
\( \phi(X) := 2X^k - k X^2 \)
satisfies conditions 
\begin{enumerate}[(A)]
\item\label{cond1} $|\phi''(x)| \geq 1$ on $\Z_p$, and 
\item\label{cond2} the second order difference polynomial $\psi(x,Y,z)$ of $\phi$ splits completely over $\Z_p$ for all $x, z \in p\Z_p$. 
\end{enumerate}
\end{proposition}
In particular, whenever $p$ and $k$ satisfy the hypotheses of Proposition \ref{prop:JB}, the curve $\gamma(T)=(T, 2T^k-kT^2)$ is non-degenerate in $\Z_p$, and yet its S.O.D.ing polynomial factors completely over $\Z_p$. This is in fact not a rare occurrence. Indeed, as the ensuing corollary demonstrates, the sufficient condition in the previous proposition is satisfied infinitely often as $k$ tends to infinity. 
\begin{corollary}\label{cor:sufficient_primes}
The conditions \ref{cond1} and \ref{cond2} are satisfied for $\phi(X)$, as defined in Proposition~\ref{prop:JB}, in each of the following cases:
\begin{enumerate}
\item $k=4$ and $p \equiv \pm 5 \mmod{12}$
\item $p$ is an arbitrary odd prime and $k =p + 1$. 
\end{enumerate}
Moreover, for every $k \geq 3$ such that $k-1$ is square-free, the polynomial $\phi(X)$ satisfies \ref{cond1} and \ref{cond2} modulo infinitely many primes $p$, provided we assume the validity of the Generalized Riemann Hypothesis.
\end{corollary}

\begin{proof}[Proof of Corollary~\ref{cor:sufficient_primes} from Proposition~\ref{prop:JB}]
The condition that $k-1$ be a non-$(k-2)$-power residue is equivalent to the congruence 
\[ (k-1)^{\frac{p-1}{k-2}} \not\equiv 1 \mmod{p} .\]
When $k=4$, the requirement is that $3$ be a quadratic non-residue modulo $p$, which is satisfied whenever $p \equiv \pm 5 \mmod{12}$. 
Similarly, when $k= p + 1$, the congruence above is trivially satisfied. Finally, the last statement follows directly from Theorem~1.1 in \cite{Moree}.
\end{proof}

\begin{remark}The remaining case when $k-1$ is divisible by a square is discussed in greater detail in Theorem 5.1 of \cite{Moree}.
\end{remark}

To prove Proposition~\ref{prop:JB}, we need to utilize some properties of S.O.D.ing polynomials. We record these in the following lemma. 
\begin{lemma}\label{diff-poly}
Let $\phi \in \Z[X]$ be a polynomial of degree $k$, defined by $\phi(X) := \sum_{j=0}^{k} a_j X^j$. If $k \geq 2$, then 
\begin{equation}\label{eq:diff-poly}
\psi(0,Y,0) 
= \sum_{j=0}^{k-2}(j+2)a_{j+2}Y^j 
,\end{equation}
and $\psi(0,Y,0)=0$ for $k=0, 1$.
\end{lemma}

\begin{proof}[Proof of Lemma~\ref{diff-poly}]

It suffices to prove the identity \eqref{eq:diff-poly} for monomials $\phi(X):=X^k$ for $k \in \N_0$. 
The cases $k=0$ or 1 are easily verified by inspection, so henceforth we assume that $k \geq 2$. 

For a univariate polynomial $\phi(T) \in \lf[T]$ we define the first order difference polynomial $\chi \in \lf[X,Y]$ to be the continuous bivariate polynomial satisfying the relation 
\begin{equation}\label{def:FOD}
{\phi(X)-\phi(Y)} = (X-Y)\chi(X,Y) 
.\end{equation}	
Observe that $\chi(X,X) = \phi'(X)$. When $\phi(X)=X^k$ this can be seen from \eqref{def:FOD} upon applying the Binomial Theorem to $\phi((X-Y)+Y)-\phi(Y)$, and it follows for a general polynomial by linearity. 

Recall the definition \ref{def:SOD}. A modicum of computation shows that
\begin{align*}
(Y-Z)\psi(X,Y,Z) 
&= \frac{\phi(X+Y-Z)-\phi(Y)}{X-Z} -\frac{\phi(X)-\phi(Z)}{X-Z} 
\\ &= \chi(Y,Y+X-Z)-\chi(X,Z) 
.\end{align*}
Taking $X=Z$, we obtain 
\begin{align*}
\psi(X,Y,X) 
&= \frac{\chi(Y,Y)-\chi(X,X)}{Y-X} = \frac{\phi'(Y)-\phi'(X)}{Y-X}
,\end{align*}
whence we see that 
\begin{align*}
\psi(0,Y,0)
= \frac{\phi'(Y)-\phi'(0)}{Y} 
.\end{align*}
Plugging in $\phi(X)=X^k$ for $k \geq 2$, we find that $\psi(0,Y,0)=kY^{k-2}$ as claimed. 
\end{proof}

\begin{proof}[Proof of Proposition \ref{prop:JB}]
We start by proving condition \ref{cond1}. Take $\phi(X)$ as in the statement of the proposition, and compute that 
\[\phi''(X) = 2k[(k-1)X^{k-2} - 1].\] 
If $|\phi''(x)|_{p} < 1$ for some $x \in \rroi$, then $\phi''(x) \equiv 0 \mmod p$. Since $p \nmid 2k$, this implies that we must have $(k-1)X^{k-2} \equiv 1 \mmod p$. This is clearly impossible when $p=k-1$. For larger $p$ it follows from the hypothesis of the proposition that $k-1$ is not a $(k-2)$-th power residue modulo $p$, so neither is its inverse $(k-1)^{-1}$. Consequently, this congruence does not have any solutions in $\Z/p\Z$, and thus no solutions in $\rroi$. 

To establish part \ref{cond2}, we begin by deducing from Lemma~\ref{diff-poly} that the S.O.D.ing polynomial $\psi(X,Y,Z)$ satisfies $\psi(0,Y,0) = 2k(Y^{k-2} -1)$. By assumption we have $p \equiv 1 \mmod{k-2}$ which implies that there exists some integer $h$ such that $(p-1)/h=k-2$. Fix a primitive root $\zeta \in \Z/p\Z$. The congruence $Y^{k-2} \equiv 1 \mmod p$ has distinct solutions $\zeta^{jh}$ for $1 \leq j \leq k-2$, and consequently we have the factorization 
\[
Y^{k-2} -1 = \prod_{j=1}^{k-2} (Y-\zeta^{jh}). 
\]
Thus, we see that \ref{cond2} is satisfied for $x=z=0$. 

Next, we show that \ref{cond2} is also satisfied for a large family of choices $x,z \in \rroi$ with $x \neq z$. Fix any $x, z \in p\rroi$. Then by our above considerations, the polynomial $\psi(x,Y,z)$ satisfies 
\[
\psi(x,Y,z) \equiv \psi(0,Y,0) \equiv 2k(Y^{k-2} -1) \equiv 2k \prod_{j=1}^{k-2} (Y-\zeta^{jh}) \mmod{p}.
\]
In other words, the polynomial $\psi(x,Y,z)$ splits completely over $\Z/p\Z$, and all factors are distinct. It follows that 
\[
\partial_Y \psi(x,Y,z)|_{Y=\zeta^{jh}} 
\not\equiv 0 \mmod{p} \qquad \text{for all $1 \leq j \leq k-2$}
.\]    
We can therefore apply Hensel's lemma and lift the roots $\zeta^{jh} \in \Z/p \Z$ of $\psi(x,Y,z)$ uniquely to roots $\xi_1, \ldots, \xi_{k-2}$ in $\rroi$. This shows that for any choice of $x, z \in p\rroi$ the polynomial $\psi(x,Y,z)$, viewed as a polynomial in $Y$ only, splits completely, which proves \ref{cond2}.  
\end{proof}

\bibliographystyle{amsalpha}

\end{document}